\documentclass{amsart}
\usepackage{graphicx}
\newtheorem{thm}{Theorem}[section]
\newtheorem{cor}[thm]{Corollary}
\newtheorem{lem}[thm]{Lemma}

\theoremstyle{definition}

\theoremstyle{remark}

\numberwithin{equation}{section}

\newcommand{\F}{\mathcal{F}}
\newcommand{\cP}{\mathcal{P}}
\newcommand{\R}{\mathbb{R}}
\newcommand{\B}{\mathcal{B}}
\begin{document}

\title[]{Maximal  multiplier operators in $L^{p(\cdot)}(\mathbb{R}^{n})$ spaces}%
\author{Amiran Gogatishvili and Tengiz Kopaliani}%

\address{Amiran Gogatishvili \\
Institute of Mathematics of the Academy of Sciences of the Czech Republic \\
\'Zitna 25 \\
115 67 Prague 1, Czech Republic} \email{gogatish@math.cas.cz}

\address{Tengiz Kopaliani \\
Faculty of Exact and Natural Sciences\\
I. Javakhishvili Tbilisi State University\\
 University St. 2\\
 0143 Tbilisi, Georgia}
\email{tengiz.kopaliani@tsu.ge}


\thanks{ The research was in part supported by the grants  no. 13/06 and no. 31/48 of the Shota Rustaveli National Science Foundation.
 The research of A.Gogatishvili was partially supported by the grant P201/13/14743S of the Grant agency of the Czech Republic and RVO: 67985840.}

\begin{abstract}
In this paper we study some estimates of norms in variable exponent
Lebesgue spaces for maximal multiplier operators.We will consider
the case when multiplier  is the Fourier transform of a compactly
supported Borel measure.

\end{abstract}
\maketitle
\section{Introduction}

Given a multiplier $m\in L^{\infty}(\mathbb{R}^{n}),$ we define the
operators  $M_{t},\,t>0$ by
$(M_{t}f)^{\wedge}(\xi)=\widehat{f}(\xi)m(t\xi)$ and the maximal
multiplier operator
$$
\mathcal{M}_{m}f(x):=\sup\limits_{t>0}|(M_{t}f)(x)|
$$
which is well defined a priori for a Schwartz function $f$ in
$S(\mathbb{R}^{n})$.

It is well known, if multiplier $m$  satisfies well known
Mikhlin-H\"{o}mander
 condition
 $$
 |\partial^{\alpha}m(\xi)|\leq C_{\alpha}|\xi|^{\alpha}
 $$
 for all (or sufficiently large) multiindices $\alpha$, then the multiplier operator
 $f \mapsto \F^{-1}[m\widehat f]$ is bounded in
 $L^{p}(\mathbb{R}^{n})$ when $1<p<\infty$ (see \cite{H}, \cite{M}). 
  Note that maximal operator $\mathcal{M}_{m}$
 formed by multiplier $m$ with  Mikhlin-H\"{o}mander condition  in general
 not bounded on $L^{p}(\mathbb{R}^{n}).$ The corresponding example
 can be find in \cite{CGS}.

 We will consider the case when multiplier $m$ is the Fourier
transform of a compactly supported Borel measure. In this case the
operator $M_{t},\,t>0$ we can represent as a convolution operator
$$
M_{t}f(x)=\int_{S}f(x-ty)d\sigma(y),
$$
where $\sigma$ is a compactly supported Borel measure on the set
$S\subset\mathbb{R}^{n}$ and $\widehat{\sigma}(\xi)=m(\xi).$
Obviously we have
$$
\mathcal{M}_{m}f(x)\equiv\mathcal{M}_{S}f(x):=\sup\limits_{t>0}\left|\int_{S}f(x-ty)d\sigma(y)\right|.
$$

We say that $\sigma$ is locally uniformly $\beta$-dimensional
($\beta>0$) if $\sigma(\mbox{B}(x,R))\leq C_\beta R^{\beta},$  where
$\mbox{B}(x,R)$ is a ball of radius $R\leq1$ centered at $ x.$
It is easy to see that a locally uniformly $\beta$-dimensional measure
must be absolutely continuous with respect to $\beta$-dimensional Hausdorff
measure $\mu_\beta$, but such a measure need not exhibit any actual "fractal"
behavior. Thus, for example, Lebesgue measure is locally uniformly
$\beta$-dimensional for any $\beta < n$. We can allow $\beta = 0$ in these definitions, in
which case a measure is uniformly $0$-dimensional if and only if it is finite,
and locally uniformly $O$-dimensional if and only if $\sigma(B_1(x))$ is uniformly
bounded in $x$.

Rubio de Francia \cite{R} proved following
\begin{thm}\label{th1.1}
 If $m(\xi)$  is the Fourier transform of a compactly supported
 Borel measure and satisfies $|m(\xi)|\leq(1+|\xi|)^{-a}$ for some
 $a>1/2$ and all $\xi\in\mathbb{R}^{n},$ then the maximal operator
 $\mathcal{M}_{m}$ maps $L^{p}(\mathbb{R}^{n})$ to itself when
 $p>\frac{2a+1}{2a}.$
\end{thm}

The case when  $\sigma$ is  normalized surface measure  on   the
$(n-1)$-dimensional unit sphere was investigated by Stein \cite{St}.
  According to  Stein's theorem for
corresponding maximal operator (spherical maximal operator)
$$
\|\mathcal{M}_{S}f\|_{L^{p}(\mathbb{R}^{n})}\leq
C_{p}\|f\|_{L^{p}(\mathbb{R}^{n})}
$$
holds if $p>n/(n-1),\,n\geq3,$ where $f$ is initially taken to be in
the class of rapidly decreasing functions. The two-dimensional
version of this result was proved by Burgain \cite{B}. The key
feature of the spherical maximal operator is the non-vanishing
Gaussian curvature of the sphere. Indeed, one obtains the same
$L^{p}$ bounds if the sphere is replaced by a piece of any
hypersurface in $\mathbb{R}^{n}$ with everywhere non-vanishing
Gaussian curvature (see \cite{Gr}).

Note that for normalized surface measure on the sphere we have
$|\widehat{\sigma}(\xi)|\leq C(1+|\xi|)^{-(n-1)/2}$ and from Theorem
Rubio de Francia  follows Stein's theorem on boundedness spherical
maximal operator in $L^{p}(\mathbb{R}^{n})$ (see \cite{R}).  More
generally, if  $\sigma$ is smooth compactly supported measure in a
hypersurface on $\mathbb{R}^{n}$ with $k$ non vanishing principal
curvatures $(k>1),$ then $|\widehat{\sigma}(\xi)|\leq
C(1+|\xi|)^{-k/2}$ and  from Theorem Rubio de Francia  follows
Greenleaf's  theorem ( see \cite{Gr}, \cite{R}).

The main tool used in proving Rubio de Francia's  maximal theorems
is the square function technique. Essentially, this says that if the
Fourier transform $m(\xi)$ of a compactly supported Borel measure
$\sigma$ has decay of order $-1/2-\epsilon;\,\epsilon>0$ i.e.,
\begin{equation} \label{1.1}
|m(\xi)|\leq C(1+|\xi|)^{-1/2-\epsilon}
\end{equation}
then the maximal  operator $\mathcal{M}_{m}$
 is bounded on $L^{2}.$  A modified proof of this results due by Iosevich and Sawyer
 (See Theorem 15 in \cite{IoSa1}) shows that
 the (1.1) condition can be
replaced by  more generally  conditions
$$
\left\{\int_{1}^{2}|m(t\xi)|^{2}dt\right\}^{1/2}\leq
C(1+|\xi|)^{-1/2}\gamma(|\xi|),
$$
$$
\left\{\int_{1}^{2}|\nabla m(t\xi)|^{2}dt\right\}^{1/2}\leq
C(1+|\xi|)^{-1/2}\gamma(|\xi|),
$$
where $\gamma$ is bounded and nonincreasing on $[0,\infty),$ and
$\sum_{n=0}^{\infty}\gamma(2^{n})<\infty.$

Our aim of this paper is to study boundedness properties of the Rubio
de Francia's maximal multiplier operator $\mathcal{M}_{m}$ in
variable Lebesgue spaces.

The boundedness of the spherical maximal operator in
variable Lebesgue spaces was studied in the  papers \cite{FGK} and \cite{FGK1}.


\section{The main results}

The Lebesgue spaces $L^{p(\cdot)}(\mathbb{R}^{n})$ with variable
exponent  and the corresponding variable Sobolev spaces
$W^{k,p(\cdot)}(\mathbb{R}^{n})$ are of interest for their
applications to modeling problems in physics, and to the study of
variational integrals and partial differential equations with
non-standard growth condition (see \cite{DHHR}, \cite{CUF}).

We define $\cP(\R^n)$ to be the set of all measurable functions $p: \mathbb{R}^{n} \to[1,\infty]$.
Functions $p\in \cP(\R^n)$ are called variable exponents on $\R^n$. We define $p^-=\mbox{essinf}_{x\in \mathbb{R}^{n}}p(x)$ and $p^+=\mbox{esssup}_{x\in\mathbb{R}^{n}}p(x)$.
If $p^+ < \infty$, then we call $p$ a bounded variable exponent.


Let $p\in \cP(\R^n)$,  $L^{p(\cdot)}(\mathbb{R}^{n})$ denotes the set of
measurable functions $f$ on $\mathbb{R}^{n}$ such that for some
$\lambda>0$
$$
\int_{\mathbb{R}^{n}}\left(\frac{|f(x)|}{\lambda}
\right)^{p(x)}dx<\infty.
$$
This set becomes a Banach function space when equipped with the norm
$$
\|f\|_{p(\cdot)}=\inf\left\{\lambda>0:\,\,\int_{\mathbb{R}^{n}}\,
\left(\frac{|f(x)|}{\lambda}\right)^{p(x)}dx\leq 1 \right\}.
$$

Let $B(x,r)$ denote the open ball in $\mathbb{R}^{n}$ of radius $r$
and center $x.$ By $|B(x,r)|$ we denote $n-$dimensional Lebesgue
measure of $B(x,r).$  The Hardy-Littlewood maximal operator $M$ is
defined on locally integrable function $f$ on $\mathbb{R}^{n}$ by
the formula
$$
Mf(x)=\sup_{r>0}\frac{1}{|B(x,r)|}\int_{B(x,r)}|f(y)|dy.
$$

 In many applications a crucial step has
been to show that Hardy-Littlwood maximal operator is bounded on a
variable $L^{p(\cdot)}$ spaces. Note that many classical operators in
harmonic analysis such as singular integrals, commutators and
fractional integrals are bounded on the variable Lebesgue space
$L^{p(\cdot}(\mathbb{R}^{n})$ whenever the Hardy-Littlewood maximal
operator is bounded on $L^{p(\cdot)}(\mathbb{R}^{n}).$

 Let $\mathcal{B}(\mathbb{R}^{n})$ be the class of all functions $p \in \cP(\mathbb{R}^{n})$ for which the
 Hardy-Littlewood maximal operator $M$ is bounded on $L^{p(\cdot)}(\mathbb{R}^{n}).$
  This class has been  a focus of intense study in recent years. We refer to the books  \cite{DHHR}, \cite{CUF},
where several results on maximal, potential and singular integral operators
in variable Lebesgue spaces were obtained.

We say that a function $p: \mathbb{R}^{n} \to(0,\infty)$ is locally
log-H\"older continuous on $\mathbb{R}^{n}$ if there exists $c_1 >
0$ such that
$$
|p(x)-p(y)|\le c_1 \frac{1}{\log(e + 1/|x- y|)}
$$
for all $x, y\in \mathbb{R}^{n}$, $|x-y|,1/2$.  We say that $p(\cdot)$ satisfies
the log-H\"older decay condition if there exist $p_\infty\in
(0,\infty)$ and a constant $c_2 > 0$ such that
$$
|p(x)-p_\infty|\le c_2 \frac{1}{\log(e + |x|)}
$$
for all $x\in \mathbb{R}^{n}$. We say that $p(\cdot)$ is globally
log-H\"older continuous in $\mathbb{R}^{n}$ ($ p(\cdot)\in
\mathcal{P}_{\log}$) if it is locally log-H\"older continuous and
satisfies the log-H\"older decay condition.

 If
$p:\mathbb{R}^{n}\to (1,\infty)$ is globally log-H\"older continuous
function in $\mathbb{R}^{n}$ and $p^->1$, then the classical
boundedness theorem for the Hardy-Littlewood maximal operator can be
extended to $L^{p(\cdot)}$ (see \cite{CUFN1, CUFN2, CUF, DHHR}).

Our main results are the following
\begin{thm} \label{th2.1} Let $p(\cdot)\in \cP(\R^n)$.
Let  $m(\xi)$  is the Fourier transform of a compactly supported
 Borel measure $\sigma$  and  the following conditions are fulfilled:

 $1)\,\,\sigma$  is locally  $\beta$-dimensional, where $0\leq\beta\leq n$;

$ 2)\,\,  \left\{\int_{1}^{2}|m(t\xi)|^{2}dt\right\}^{1/2}\leq
C(1+|\xi|)^{-\alpha}, $

$3)\,\,\left\{\int_{1}^{2}|\nabla m(t\xi)|^{2}dt\right\}^{1/2}\leq
C(1+|\xi|)^{-\alpha}, $

where $\alpha>1/2.$ If
$\frac {2\theta p(\cdot)}{2-(1-\theta)p(\cdot)}\in \mathcal{B}(\mathbb{R}^{n})$ for some
$0<\theta< \frac{2\alpha-1}{2\alpha-1+2n-2\beta},$ then the maximal
operator $\mathcal{M}_{m}$ maps $L^{p(\cdot)}(\mathbb{R}^{n})$ to
itself.
\end{thm}

 \begin{thm} \label{th2.2}
Let  $m(\xi)$  is the Fourier transform of a compactly supported
 Borel measure $\sigma$  and  the following conditions are fulfilled:

 $1)\,\,\sigma$  is locally  $\beta$-dimensional, where $0\leq\beta\leq n$;

$ 2)\,\,  \left\{\int_{1}^{2}|m(t\xi)|^{2}dt\right\}^{1/2}\leq
C(1+|\xi|)^{-\alpha}, $

$3)\,\,\left\{\int_{1}^{2}|\nabla m(t\xi)|^{2}dt\right\}^{1/2}\leq
C(1+|\xi|)^{-\alpha}, $

where $\alpha>1/2.$  If $p(\cdot)\in \mathcal{P}_{\log}$ and
$$
\frac{2n+2\alpha-2\beta-1}{n+2\alpha-\beta-1}<p_{-}\leq
p_{+}<\frac{2n+2\alpha-2\beta-1}{n-\beta},
$$
 then the maximal operator
$\mathcal{M}_{m}$ maps $L^{p(\cdot)}(\mathbb{R}^{n})$ to itself.
\end{thm}

 \begin{thm} \label{th2.3}
Let  $m(\xi)$  is the Fourier transform of a compactly supported
 Borel measure $\sigma$  and  the following conditions are fulfilled:

 $1)\,\,\sigma$  is locally  $\beta$-dimensional, where $0\leq\beta\leq n$;

$ 2)\,\,  \left\{\int_{1}^{2}|m(t\xi)|^{2}dt\right\}^{1/2}\leq
C(1+|\xi|)^{-\alpha}, $

$3)\,\,\left\{\int_{1}^{2}|\nabla m(t\xi)|^{2}dt\right\}^{1/2}\leq
C(1+|\xi|)^{-\alpha}, $

where $\alpha>1/2.$ If $p(\cdot)\in \mathcal{P}_{\log}$ and
$$
\frac{2n+2\alpha-2\beta-1}{n+2\alpha-\beta-1}<p_{-}\leq
p_{+}<\frac{n+2\alpha-\beta-1}{n-\beta}p_{-},
$$
 then the maximal operator
$\mathcal{M}_{m}$ maps $L^{p(\cdot)}(\mathbb{R}^{n})$ to itself.
\end{thm}

If we take $\beta=0, $ then we will  obtain  analog of Theorem Rubio
de Francia for variable exponent Lebesgue spaces.
\begin{cor}\label{cor2.4}
If $m(\xi)$  is the Fourier transform of a compactly supported
 Borel measure and satisfies $|m(\xi)|\leq(1+|\xi|)^{-a}$ for some
 $a>1/2$ and all $\xi\in\mathbb{R}^{n}.$ If $p(\cdot)\in \mathcal{P}_{\log}$ and
$$
\frac{2n+2\alpha-1}{n+2\alpha-1}<p_{-}\leq
p_{+}<\frac{n+2\alpha-1}{n}p_{-}.
$$
 then the maximal operator
$\mathcal{M}_{m}$ maps $L^{p(\cdot)}(\mathbb{R}^{n})$ to itself.
\end{cor}

Let $\mu$ denote a Hausdorff measure on $E\subset[0,1]$  and
$\nu_{r}$ denote the rotationally invariant probability measure on
the sphere of radius $r.$ Let
\begin{equation} \label{2.1}
\sigma=\int_{0}^{1}\nu_{r}d\mu(r)
\end{equation}
denote the corresponding rotationally invariant measure on the set
$E_{n}=\{x\in\mathbb{R}^{n}:\,\,|x|\in E\}.$

If  measure ${\mu}$ is locally $\alpha$-dimensional
$(0\leq\alpha\leq1)$, then  measure $\sigma$ is locally uniformly
$n-1+\alpha$-dimensional and
$$
\left(\int_{1}^{2}|\widehat{\sigma}(t\xi)|^{2}dt\right)^{1/2}\leq
C(1+|\xi|)^{-\frac{n-1+\alpha}{2}},
$$
moreover, the same estimates hold if $\widehat{\sigma}(t\xi)$ is
replaced by $\nabla\widehat{\sigma}(t\xi)$ (see \cite{IoSa2}).

Let  $E$ denote the Cantor-like  subset of $[0,1]$ consisting  of
real numbers whose base $m,\,m>2,$ expansions have only $0'$s and
$1'$s. Let $\mu$  denote the probability measure  on $E.$ Note that
$\widehat{\mu}(\xi)$  does not tend to $0$ as $\xi\rightarrow\infty$
(see e.g. \cite{Z}) and for corresponding measure $\sigma$ Fourier
transform $\widehat{\sigma}(\xi)$ decays only of order
$-\frac{n-1}{2}$ at infinity, but  square function
$$
\left(\int_{1}^{2}|\widehat{\sigma}(t\xi)|^{2}dt\right)^{1/2}
$$
decays of order $-\frac{n-1+\alpha}{2},$  where
$\alpha=\frac{\log2}{\log m}$ is dimension of $E.$ (see
\cite{IoSa2}).

\begin{cor}\label{cor2.5}
Let $\mu$ denote a Hausdorff measure on $E\subset[0,1].$ Suppose
$\mu$ is locally $\alpha$-dimensional $0\leq\alpha<1$. Let
$\mathcal{M}_{\sigma}$ maximal operator corresponding (2.1)
measure.if $p(\cdot)\in \mathcal{P}_{\log}$  and
$$
\frac{n-\alpha}{n-1}<p_{-}\leq p_{+}<\frac{n-1}{1-\alpha}p_{-},
$$
then the maximal operator $\mathcal{M}_{\sigma}$ maps
$L^{p(\cdot)}(\mathbb{R}^{n})$ to itself.
\end{cor}

\section{Proof of main results}

{\sl Proof of Theorem~\ref{th2.1}.}  We set $m(\xi)=\widehat{d\sigma}(\xi).$
Obviously $m(\xi)$ is a $C^{\infty}$ function.  To study the maximal
multiplier operator $\mathcal{M}_{m}f(x)$ we decompose the
multiplier $m(\xi)$ into radial pieces as follows: we fix a radial
$C^{\infty}$ function $\varphi_{0}$ in $\mathbb{R}^{n}$ such that
$\varphi_{0}(\xi)=1$ when $|\xi|\leq1$ and $\varphi_{0}(\xi)=0$ when
$|\xi|\leq2.$ For $j\geq1$ we let
$$
\varphi_{j}(\xi)=\varphi_{0}(2^{-j}\xi)-\varphi_{0}(2^{1-j}\xi)
$$
and we observe that $\varphi_{j}$ is localized near $|\xi|\approx
2^{j}.$ Then we have
$$
\sum_{j=0}^{\infty}\varphi_{j}=1.
$$
Set $m_{j}=\varphi_{j}m$ for all $j\geq0.$ Then $m_{j}$ are
$C_{0}^{\infty}$ functions that satisfy
$$
m=\sum_{j=0}^{\infty}m_{j}.
$$
Also, the following estimate is valid:
$$
\mathcal{M}_{m}f\leq \sum_{j=0}^{\infty}\mathcal{M}_{j}f
$$
where
$$
\mathcal{M}_{j}f(x)=\sup\limits_{t>0}|\F^{-1}\left(\widehat{f}(\xi)m_{j}(t\xi)\right)(x)|.
$$

 Note that for any $j\geq0$ we
have (see proof of Theorem 15 in \cite{IoSa1}) the estimate
\begin{equation}\label{3.1}
\|\mathcal{M}_{j}f\|_{L^{2}}\leq C 2^{(1/2-a)j}\|f\|_{L^{2}}
\end{equation}
for all $f\in L^{2}(\mathbb{R}^{n}).$

Note also that since  $\widetilde{p}(\cdot):=\frac {2\theta p(\cdot)}{2-(1-\theta)p(\cdot)}\in \mathcal{B}(\mathbb{R}^{n})$  we have  the estimate
\begin{equation} \label{3.2}
\|\mathcal{M}_{j}f\|_{\widetilde{p}(\cdot)}\leq C
2^{j(n-\beta)}\|f\|_{\widetilde{p}(\cdot)}
\end{equation}
for any  $j\geq0.$
 The proof of estimate \eqref{3.2} is based on the
estimate
\begin{equation} \label{3.3}
\mathcal{M}_{j}f(x)\leq C 2^{j(n-\beta)}Mf(x),
\end{equation}
where $M$ is Hardy-Littlewood maximal operator.

The proof of \eqref{3.3} for specific  measure defined by \eqref{2.1}  was
done in \cite{IoSa2}. The proof based only on the geometric
assumption of the measure (assumption 1)). We will prove it for
general case
 for completeness.

 To establish \eqref{3.3}, it is suffices to show that for any $M>n$ there is a constant $C_M <\infty$ such that
\begin{equation} \label{3.4}
\left|\left(\F^{-1}(\varphi_{j})\ast d\sigma\right) (x)\right|\leq
\frac{C2^{j(n-\beta)}}{(1+|x|)^{M}.}
\end{equation}

 Using the fact that $\varphi$ is a Schwartz function,
we have for every $N > 0$,

\begin{equation}\label{3.5}
\left|\left(\F^{-1}(\varphi_{j})\ast d\sigma\right)(x)\right|
\leq C_N 2^{nj}\int_{\mathbb{R}^{n}}\frac{d\sigma(y)}{(1+2^{j}|x-y|)^{N}}.
\end{equation}

Let  $N >M$. We split the last integral into the regions
\begin{align*}
S_{-1}(x) =  \{y\in  \R^n : 2^j|x-y|\le 1\}\\
\intertext{and for $k > 0$,}
S_k(x) = \{y \in \R^n : 2^k < 2^j|x-y|\le 2^{k+1}\}.
\end{align*}
We obtain the following estimate for the expression $\left|\left(\F^{-1}(\varphi_{j})\ast d\sigma\right)
(x)\right|$
\begin{align}\label{3.6}
& \sum_{k=-1}^{j}\int_{S_k(x)}\frac{C_N 2^{nj}d\sigma(y)}{(1+2^{j}|x-y|)^{N}}+
\sum_{k=j+1}^{\infty}\int_{S_k(x)}\frac{C_N 2^{nj}d\sigma(y)}{(1+2^{j}|x-y|)^{N}}\\
&\le C^{\prime}_N 2^{nj} \sum_{k=-1}^{j}\frac{\sigma(S_k(x))\chi_{{}_{B(0,3)}}(x)}{2^{kN}} +C_N 2^{nj}\sum_{k=j+1}^{\infty}
\frac{\sigma(S_k(x))\chi_{{}_{B(0,2^{k+1-j}+1)}}(x)}{2^{kN}}\notag\\
&=:I+II.\notag
\end{align}
 Using the fact that
$\sigma$ is locally uniformly $\beta$-dimensional, together with the
fact that for $y\in S_k(x)$ we have $|x|\le 2^{k+1-j} +1$, we obtain the following estimate
\begin{equation}\label{3.7}
I\le  C^{\prime}_N 2^{nj}\sum_{k=-1}^{j}\frac{C_\beta 2^{(k+1-j)\beta}\chi_{{}_{B(0,3)}}(x)}{2^{kN}}\le 
C_{N,\beta} 2^{(n-\beta)j}\chi_{{}_{B(0,3)}}(x).
\end{equation}

On the other hand

\begin{align}\label{3.8}
II&\le  C^{\prime}_N 2^{nj}\sum_{k=j+1}^{\infty}C2^{-kN}\chi_{{}_{B(0,2^{k+1-j}+1)}}(x)\\
 &\le  C^{\prime}_N\sum_{k=j+1}^{\infty}2^{nj}2^{-kN}\frac{(1+2^{k-j+2})^{M}}{(1+|x|)^{M}}\notag\\
 &\leq  C^{\prime}_M\sum_{k=j+1}^{\infty}\frac{2^{(k-j)(M-N)}}{2^{k(N+1-n)}}\notag \\
 &\leq \frac{C^{\prime\prime}_M 2^{j}}{(1+|x|)^{M}},\notag
\end{align}
where we used that $N > M > n$.
From \eqref{3.5}-\eqref{3.8} we obtain \eqref{3.4} and consequently \eqref{3.3}.

Note that
$$\frac{1}{p(\cdot)}=\frac{1-\theta}{2}+ \frac{\theta}{\widetilde{p}(\cdot)},$$ 
and, therefore 
$$L^{p(\cdot)}(\mathbb{R}^{n})=[L^{2}(\mathbb{R}^{n}),L^{\widetilde{p}(\cdot)}(\mathbb{R}^{n})
 ]_{\theta},$$ 
 (where $[X_0,X_1]_\theta$ is a complex interpolation space).
 Now from \eqref{3.1}-\eqref{3.2} we obtain
\begin{equation} \label{3.9}
\|\mathcal{M}_{j}\|_{L^{p(\cdot)}\rightarrow L^{p(\cdot)}}\leq C
\|\mathcal{M}_{j}\|^{1-\theta}_{L^{2}\rightarrow
L^{2}}\,\|\mathcal{M}_{j}\|_{L^{\widetilde{p}(\cdot)}\rightarrow
L^{\widetilde{p}(\cdot)}}\le C^{\prime}
2^{(1/2-\alpha)(1-\theta)j}2^{j(n-\beta)\theta}.
\end{equation}

Using the last estimate we obtain if
$0<\theta<\frac{2\alpha-1}{2\alpha-1+2n-2\beta},$ then
$$
\|\mathcal{M}_{m}\|_{p(\cdot)}\le C^{\prime}
\sum_{j=0}^{\infty}2^{(1/2-a)(1-\theta)j}2^{j(n-\beta)\theta}\|f\|_{p(\cdot)}\le C^{\prime\prime}
 \|f\|_{p(\cdot)}.
$$
$\Box$

To prove Theorem~\ref{th2.2} we need the following lemma.

\begin{lem} \label{lm3.1}
Suppose $\alpha>1/2,\, 0\leq\beta\le n$ and for exponent
$p:\mathbb{R}^{n}\rightarrow(1,+\infty)$ we have
$$
\frac{2n+2\alpha-2\beta-1}{n+2\alpha-\beta-1}<p_{-}\leq
p_{+}<\frac{2n+2\alpha-2\beta-1}{n-\beta}.
$$
Then there exists exponent
$\widetilde{p}:\mathbb{R}^{n}\rightarrow(1,+\infty)$ such that
$1<\widetilde{p}_{-}\leq\widetilde{p}_{+}<\infty$ and
$\frac{1}{p(x)}=\frac{1-\theta}{2}+\frac{\theta}{\widetilde{p}(x)};\,\,x\in\mathbb{R}^{n}$
for some $\theta$ with property
$0<\theta<\frac{2\alpha-1}{2n+2\alpha-2\beta-1}.$
\end{lem}
\begin{proof}
Note that if $\beta<n$ then
$$
1<\frac{2n+2\alpha-2\beta-1}{n+2\alpha-\beta-1}<2<\frac{2n+2\alpha-2\beta-1}{n-\beta},
$$
and if $\beta=n,$  then
$$
\frac{2n+2\alpha-2\beta-1}{n+2\alpha-\beta-1}=0\,\,\,\,
\mbox{and}\,\,\,\frac{2n+2\alpha-2\beta-1}{n-\beta}=\infty.
$$

 We have
$$
\frac{n-\beta}{2n+2\alpha-2\beta-1}<\inf_{x\in
\mathbb{R}^{n}}\frac{1}{p(x)}\leq \sup_{x\in
\mathbb{R}^{n}}\frac{1}{p(x)}<\frac{n+2\alpha-\beta-1}{2n+2\alpha-2\beta-1}.
$$

Let $\frac{1}{p(x)}=\frac{1}{2}+r(x)$. By assumption we have
\begin{equation} \label{3.10}
\frac{n-\beta}{2n+2\alpha-2\beta-1}-\frac{1}{2}<\inf_{x\in\mathbb{R}^{n}}r(x)\leq\sup_{x\in\mathbb{R}^{n}}r(x)<
\frac{n+2\alpha-\beta-1}{2n+2\alpha-2\beta-1}-\frac{1}{2}.
\end{equation}

It is easy to see that the equation
\begin{equation}\label{3.11}
\frac{1}{p(x)}=\frac{1-\theta}{2}+\frac{\theta}{\widetilde{p}(x)};
\end{equation}
 is equivalent to
 \begin{equation} \label{3.12}
\frac{1}{2}+\frac{r(x)}{\theta}=\frac{1}{\widetilde{p}(x)}.
\end{equation}
Using \eqref{3.9} we may  take small $\delta>0$  such that
$$
\frac{n-\beta}{2n+2\alpha-2\beta-1}-\frac{1}{2}+\delta<\inf_{x\in\mathbb{R}^{n}}r(x)\leq\sup_{x\in\mathbb{R}^{n}}r(x)<
\frac{n+2\alpha-\beta-1}{2n+2\alpha-2\beta-1}-\frac{1}{2}-\delta.
$$
Then for $\theta,\,\,0<\theta<\frac{2\alpha-1}{2\alpha-1+2\beta},$
where
$\theta=\theta<\frac{2\alpha-1}{2\alpha-1+2\beta}-\theta_{0},\,\,\,\theta_{0}>0$
we have
$$
\frac{\frac{n-\beta}{2n+2\alpha-2\beta-1}-\frac{1}{2}+\delta}{\frac{2\alpha-1}{2n+2\alpha-2\beta+1}-\theta_{0}}<
\inf_{x\in\mathbb{R}^{n}}\frac{r(x)}{\theta}
\leq\sup_{x\in\mathbb{R}^{n}}\frac{r(x)}{\theta}<\frac{\frac{n+2\alpha-\beta-1}{2n+2\alpha-2\beta-1}-\frac{1}{2}-\delta}{
\frac{2\alpha-1}{2n+2\alpha-2\beta+1}-\theta_{0}}
$$

$$
-\frac{1}{2}\frac{\frac{2a-1}{2n+2a-2\beta-1}-2\delta}{\frac{2a-1}{2n+2a-2\beta-1}-\theta_{0}}<\inf_{x\in\mathbb{R}^{n}}\frac{r(x)}{\theta}
\leq\sup_{x\in\mathbb{R}^{n}}\frac{r(x)}{\theta}<\frac{1}{2}\frac{\frac{2a-1}{2n+2a-2\beta-1}-2\delta}{\frac{2a-1}{2n+2a-2\beta-1}-\theta_{0}}.
$$
If we take $\theta_{0}<2\delta$  we obtain
\begin{equation} \label{3.13}
-\frac{1}{2}<\inf_{x\in\mathbb{R}^{n}}\frac{r(x)}{\theta}\leq
\sup_{x\in\mathbb{R}^{n}}\frac{r(x)}{\theta}<\frac{1}{2}.
\end{equation}
From \eqref{3.11} and  \eqref{3.12} we get
 $$
 0<\inf_{x\in\mathbb{R}^{n}}\frac{1}{\widetilde{p}(x)}\leq\sup_{x\in\mathbb{R}^{n}}\frac{1}{\widetilde{p}(x)}<1.
 $$
Consequently we have
$1<\widetilde{p}_{-}\leq\widetilde{p}_{+}<\infty.$
\end{proof}

 \vskip+1cm  
{\sl Proof of  Theorem~\ref{th2.2}.}
Using the fact that if $p(\cdot)\in
\mathcal{P}_{\log}$ then $\widetilde{p}(\cdot):=\frac {2\theta p(\cdot)}{2-(1-\theta)p(\cdot)}\
\in \mathcal{P}_{\log}$ and   by Lemma~\ref{lm3.1} we have $1<\widetilde{p}_{-} \le\widetilde{p}_{+}<\infty$,  it is followes 
from \cite[Theorem 4.3.8]{DHHR}, that $\widetilde{p}(\cdot)\in \B(\R^n)$. Now  the proof of Theorem~\ref{th2.2} follows from Theorem~\ref{th2.1}.
$\Box$
\vskip+1cm 
{\sl Proof of  Theorem~\ref{th2.3}.} 
As by the assumption $$\frac{2n+2\alpha-2\beta-1}{(n+2\alpha-\beta-1)p_{-}}< \frac{2n+2\alpha-2\beta-1}{(n-\beta)p_{+}},$$
 we can fined  $\theta$ such that 
  $$\frac{2n+2\alpha-2\beta-1}{(n+2\alpha-\beta-1)p_{-}}<\theta< \min\left(1,\frac{2n+2\alpha-2\beta-1}{(n-\beta)p_{+}}\right).$$
It is clear, that
 $$\frac{2n+2\alpha-2\beta-1}{(n+2\alpha-\beta-1)}<\theta p_{-}<\theta p_{+} <\frac{2n+2\alpha-2\beta-1}{(n-\beta)}.$$
 It is clear that if $p(\cdot)\in\cP_{\log}$ then $\theta p(\cdot)\in\cP_{\log}$ and by Theorem~\ref{th2.2} we get that the operator $\mathcal{M}_{m}$   is bounded in $L^{\theta p(\cdot)}(\mathbb{R}^{n})$.
Using the fact that
$[L^{\infty}(\mathbb{R}^{n}),L^{p(\cdot)\theta}(\mathbb{R}^{n})]_{\theta}=L^{p(\cdot)}(\mathbb{R}^{n}),\,\,(0<\theta<1)$
and the operator $\mathcal{M}_{m}$ is bounded in
$L^{\infty}(\mathbb{R}^{n})$  and $L^{\theta p(\cdot)}(\mathbb{R}^{n})$ we obtain that operator $\mathcal{M}_{m}$ is bounded in $L^{p(\cdot)}(\mathbb{R}^{n})$. 
$\Box$

\end{document}